\documentclass[a4paper,11pt,leqno]{amsart}
\usepackage{amsmath,amscd,amsfonts,amssymb}
\usepackage{latexsym}
\usepackage{graphics}
\usepackage{epsfig}
\usepackage[english]{babel}

\numberwithin{equation}{section}
\newtheorem{theorem}{Theorem}[section]

\newtheorem{cor}[theorem]{Corollary}

\theoremstyle{definition}

\newtheorem{rem}{Remark}

\title{Conformal Geometric Inequalities on the Klein bottle}

\author{Chady El Mir}
\email{chady.mir@gmail.com}

\author{Zeina Yassine}
\email{zeina.y.yassine@gmail.com}

\address{Current adress : \newline
Department of Mathematics and computer science\newline
Beirut Arab University \newline
P.O.Box 11 - 50 - 20 Riad El Solh 11072809\newline
              Beirut, Lebanon \newline } 

\graphicspath{{images/}}

\begin{document}
\date{}

\selectlanguage{english}
\begin{abstract}
We prove three optimal conformal geometric inequalities of Blatter type (\cite{blatter}) on the Klein bottle. These inequalities provide conformal lower bounds of the volume and involve lengths of homotopy classes of curves that are candidates to realize the systole.
\end{abstract}
\medskip

\maketitle

\selectlanguage{english} \noindent {\bf Keywords} : Klein bottle, conformal metric, systole, isosystolic inequality.
\\
\newline

\section{Introduction, preliminaries and results}\label{intro}

Among all Riemannian metrics on a given compact differentiable  manifold, the most interesting ones are those that extremize some Riemannian invariant. An interesting problem, for example, is to study metrics $g$ which maximize the ratio $(sys(M,g))^n/vol(M,g)$, where $sys(M,g)$ is the systole of the $n$-dimensional Riemannian manifold $(M,g)$. Nevertheless, as in many works related to this problem, various constraints can be put on the set of Riemannian metrics considered. For example we may restrict ourselves to the set of constant curvature metrics (\cite{adams}, \cite{sab}, \cite{elmir}) or to the set of metrics conformal to a given one (\cite{blatter}, \cite{bavard2}, \cite{bavard3}, \cite{bavard5}).\\

As far as this work is concerned, we are interested in finding optimal conformal lower bounds of the volume of the Klein bottle (denoted by $K$) by (the product of) the least length of a class of non contractible curves. We will often consider (free) homotopy classes of curves that are candidates to realize the systole.  Such inequalities were proved by Blatter on the Mobius band with boundary (denoted by $M$) in 1961 (\cite{blatter}). Hoping to get a universal inequality on $M$, Blatter considers the family of curves $F$ joining two points on the boundary and going through the central geodesic of the Mobius Band at least one time. If we denote by $l^*$ the least length of a curve that belongs to $F$ and by $l_1$ the least length of a curve closed by the glide reflection generating $\pi_1(M)$, Blatter (here we follow his notations) gets the following  optimal \emph{conformal} lower bound of the volume 

$$l^*(g) l_1(g) \leq C_\beta vol(g) $$

where $C_\beta$ depends only on the conformal type of the metrics under consideration.

He was expecting to get an inequality of this type which holds for any Riemannian metric on the Mobius Band, but he found that this is not true. We will show that a similar result holds for the Klein bottle (Corollary \ref{cor1} below). Note that the family of curves $F$ is a natural set of curves on the Klein bottle since it corresponds to the homotopy class of the vertical translation in $\pi_1(K)$ (see section \ref{klein} below).

\subsection {The Klein bottle} \label{klein}

It is the quotient of
$ \mathbb{R}^2$ by
the group generated by the maps
$\sigma: z\mapsto \bar{z}+\pi$ and $t:z\mapsto z+2i\beta$. The flat metric induced by such a quotient will be denoted by $g_{\beta}$. By the uniformization theorem, we know that every Riemannian metric $g$ on the Klein bottle is conformally equivalent to a flat metric $g_\beta$ for some $\beta \in ]0,+\infty[$. The latter is called the conformal type of the metric $g$.\\

The Klein bottle satisfies the isosystolic inequality

$$   sys(K)^2 \leq \frac{\pi}{2\sqrt{2}} \cdot vol(K)$$

The equality being achieved for a spherical metric outside a singular line  (see \cite{bavard}, \cite{sakai}, \cite{elmir}). For details and (many) open problems in \emph{systolic geometry} see the book of Katz \cite{katz}, the website http://u.math.biu.ac.il/~katzmik/sgt.html, and the "deep" paper of Gromov \cite{gromov}.

We put the emphasis on the fact that the systole of any metric on the Klein bottle is realized by one of the following (free) homotopy class of curves:\\

\begin{enumerate}

\item The homotopy class of (a geodesic closed by) the glide reflection $\sigma$ of $\pi_1(K)$. 

\item The homotopy class of the vertical translation $t$ of $\pi_1(K)$.

\item The homotopy class of the horizontal translation $\sigma^2$ of $\pi_1(K)$.

\end{enumerate}

If $g$ is a Riemannian metric on $K$, the least length of a curve in the first (resp. second, resp. third) class will be denoted by $l_\sigma$ (resp. $l_v$, resp. $l_h$). Note that the extremal metric for the isosystolic inequality on the Klein bottle $(K,g_e)$ verifies $l_{\sigma}(g_e)=l_{v}(g_e)=\pi$. Its conformal type is $2\ln(\tan(\frac{3\pi}{8}))$.\\

We will often use metrics (classical in the subject) of the form $G_b=f^2(v)du^2+dv^2$ for $|u|\leq \frac{\pi}{2}$ and $|v|\leq b$, with $f$ continuous, positive, even and $b$-periodic function. The conformal type of such metric is

$$\beta=\int_0^b{\frac{dt}{f(t)}}$$

When $f$ is equal to the function $"\cos "$ in a cylinder $C_\alpha=\{|v|\leq \alpha\}$ for some $0<\alpha<\pi/2$, we denote by $\gamma_{\theta}^{a}$ the great circle going through the points $(u=\theta-\pi/2,\ v=0)$, $(u=\theta,\ v=a)$ and $(u=\theta+\pi/2,\ v=0)$.

\subsection{The results} The first result of this paper studies optimal conformal  inequalities of the form $l_v l_\sigma \leq C_\beta vol$ (we call it "Geometric Inequality of Type $\sigma-v$"), where $C_\beta$ is a constant depending only on the conformal type $\beta$. It also describes the Riemannian metrics for which the best
constant is achieved:

\begin{theorem}\label{1}
Let $b_{0}$ be such that: $b_{0}\in ]0,\frac{\pi}{2}[$ and $\tan(b_{0}) = 2b_{0}$, then for every Riemannian metric $g$ on the Klein bottle $K$ conformal to $g_{\beta}$ we have
\begin{enumerate}
	\item If $0<{\beta}\leq{2}\ln\left(\tan\left(\frac{\pi}{4}+\frac{b_{0}}{2}\right)\right)$, then
	   \begin{equation*}
l_v(g)l_\sigma(g)\leq\frac{\arcsin\left(\frac{e^{{\beta}}-1}{e^{{\beta}}+1}\right)}{\frac{e^{{\beta}}-1}{e^{{\beta}}+1}}vol(g).	
\end{equation*}
	\item If ${\beta}>{2}\ln\left(\tan\left(\frac{\pi}{4}+\frac{b_{0}}{2}\right)\right)$, then
\begin{equation*}
l_v(g)l_\sigma(g)\leq \frac{1}{2\cos(\omega)} vol(g).
\end{equation*}
where $\omega\in[b_{0},\frac{\pi}{2}[$ is given by \ \  $2\sin(\omega) = \left({\beta}-2\ln\left(\tan\left(\frac{\pi}{4}+\frac{\omega}{2}\right)\right)\right)\cos^{2}(\omega)+4\omega\cos(\omega)$.

\end{enumerate}

Moreover, the equality is achieved by a spherical metric in the first case and a flat-spherical one in the second case.
\end{theorem}

\begin{rem} The inequality in the first case was proved by Bavard in \cite{bavard5}. We complete here this study by obtaining an optimal inequality for the remaining conformal classes.

\end{rem}

\begin{cor} \label{cor1}
There does not exist any (finite) constant $c$ such that the inequality:

$$l_v(g)l_\sigma(g) \leq c \cdot vol(g)$$

holds for every Riemannian metric $g$ on the Klein bottle.

\end{cor}

The second part of the paper is devoted to establishing optimal conformal inequalities of the form $L_\sigma l_h \leq C_\beta vol$, where $L_\sigma$ denotes the least length of a curve freely homotopic to a geodesic loop closed by an element of the subgroup $<\sigma>$ generated by $\sigma$ (this is the systole in the case of the Mobius Band with boundary). We call it "Geometric inequality of type $\sigma^n-v$". Note that since the only candidates of curves for $L_\sigma$ in $<\sigma>$ are $\sigma$ and $\sigma^2$, we have $L_\sigma=\inf\{l_\sigma,l_h\}$.

\begin{theorem} \label{2}
For every Riemannian metric $g$ on the Klein bottle $K$ conformal to $g_{\beta}$, we have
\begin{enumerate}
	\item If $0<{\beta}\leq{2}\ln\left(2+\sqrt{3}\right)$, then
	   \begin{equation*}
L_\sigma(g)l_v(g)\leq\frac{\arcsin\left(\frac{e^{{\beta}-1}}{e^{{\beta}}+1}\right)}{\frac{e^{{\beta}}-1}{e^{{\beta}}+1}}vol(g).	
\end{equation*}
	\item If ${\beta}>{2}\ln\left(2+\sqrt{3}\right)$, then
\begin{equation*}
L_\sigma(g)l_v(g)\leq\frac{2}{3}.\frac{3\beta+4\pi-6\ln(2+\sqrt{3})}{4\sqrt{3}+\beta-2\ln(2+\sqrt{3})}  vol(g).
\end{equation*}
\end{enumerate}
Moreover, the equality is achieved by a spherical metric in the first case and by a flat-spherical metric in the second case.  
\end{theorem}

Since the supremum of the conformal constant $C_\beta$ over $\beta $ is equal to $2$ we have

\begin{cor}
For every Riemannian metric $g$ on the Klein bottle $K$ we have
\begin{equation*}
{L_{\sigma}(g)l_{v}(g)} \leq {2} {vol(g)}
\end{equation*}
Though this inequality is optimal, the equality case is not achieved. 

\end{cor}

In the third part we establish optimal conformal inequalities of the form $l_h l_v l_\sigma \leq C_\beta vol^{\frac{3}{2}}$ (Geometric Inequality of Type $\sigma-v-h$). This inequality has the advantage that the volume (power $\frac{3}{2}$ to get a homogeneous invariant) is bounded by all three candidates of the systole:

\begin{theorem}\label{3}
For every Riemannian metric $g$ on the Klein bottle $K$ conformal to $g_{\beta}$, we have the following inequality
\begin{equation*}
l_{\sigma}(g)l_{v}(g)l_{h}(g) \leq \frac{\sqrt{\pi}}{3\sqrt{3}}\cdot \frac{\left(b^{4}-4b\omega+\omega^{2}+\omega^{4}-2b^{2}(-2+\omega^{2})\right)^{\frac{1}{4}}(2b-\omega)}{(b-\omega)\sqrt{(b-\omega)b}}  vol(g)^{\frac{3}{2}}.	
\end{equation*}
Where $\omega\in]0,\frac{\pi}{2}[$ is given by 
$$\beta=2\ln(\tan(\frac{\pi}{4}+\frac{\omega}{2}))+\frac{2}{\cos(\omega)}\left(\tan(\omega)-\omega+\sqrt{\tan^{2}(\omega)-\omega\tan(\omega)+\omega^{2}}\right)$$ 
and $b=\tan(\omega)+\sqrt{\tan^{2}(\omega)-\omega\tan(\omega)+\omega^{2}}$. Moreover, the equality is achieved by a flat-spherical metric.
\end{theorem}

\begin{rem}

If we replace $b$ by $\tan(\omega)+\sqrt{\tan^{2}(\omega)-\omega\tan(\omega)+\omega^{2}}$ in $$C(\omega)=\frac{\left(b^{4}-4b\omega+\omega^{2}+\omega^{4}-2b^{2}(-2+\omega^{2})\right)^{\frac{1}{4}}(2b-\omega)}{(b-\omega)\sqrt{(b-\omega)b}}$$ 

we see that $C:]\frac{2}{\pi},+\infty[\rightarrow ]\frac{2\sqrt{\pi}}{3\sqrt{3}},+\infty[$ is a continuous increasing onto function and then the supremum of $C$ is $+\infty$ as $\omega\rightarrow 0$  (i.e. when $b\rightarrow 0$). Therefore we get the surprising corollary:

\end{rem}

\begin{cor} There does not exist any (finite) constant $c$ such that the inequality

$$l_v(g)l_\sigma(g) l_{h}(g) \leq c \cdot vol(g)^{\frac{3}{2}}$$

holds for every Riemannian metric $g$ on the Klein bottle $K$.

\end{cor}

\section{A maximality criterion}

The proofs of our results are based on the method of extremal
lengths initiated by Jenkins in \cite{jenk}, Gromov in \cite{gromov}, and Bavard in \cite{bavard3}.\\

Let $(M,g_e)$ be a compact Riemannian manifold, and let for each $i\in \{1,\cdots,p\}$, $S_i$ be a set of rectifiable curves, such that $S_{i}\cap S_{j}=\emptyset$ for all $ i\neq j$. We denote by $l_i(g)$ the least length of a curve in $S_i$ with respect to a Riemannian metric $g$ $M$ conformal to $g_e$.\\

For every Radon measure $\mu$ on $\Gamma=S_1\cup \cdots \cup S_p$, we associate a measure $^* \mu $ on $M$ by setting for $\varphi \in C^0(M,\mathbb{R})$:

$$ <^* \mu, \varphi>= <\mu,\overline{\varphi}>$$

where $\overline{\varphi}(\gamma)=\int{\varphi \circ \gamma (s)
ds}$, $ds$ is the arc length of $\gamma$ with respect to $g_e$.\\

%If in addition all the curves in $S_i$, $i\in \{1,\cdots,p\}$ have the same length with respect to $g_e$, then we have:

\begin{theorem} \label{extcond} (\cite{bavard2},\cite{bavard3} and \cite{jenk})

%Suppose $m_i=\mu(S_i)\neq 0$ for every $i\in \{1,\cdots, p\}$, then 
The Riemannian metric $g_e$ is maximal in its conformal class with respect to the Riemannian ratio $\frac{l_1\cdots l_p}{Vol^\frac{p}{2}}$ if there exists a positive measure $\mu$ on
$\Gamma$ such that $m_1l_1(g_{e})=m_2l_2(g_{e})=\cdots=m_pl_p(g_{e})$ and

$$^*\mu= d g_e$$

where, $m_{i}$ is the mass of the measure on $S_{i}$, and $dg_e$ is the volume measure of $(M,g_{e})$.

\end{theorem}

Since our version of theorem \ref{extcond} is not the same as in the references we gave, we give here a proof of it.

\begin{proof}

Let $g$ be a Riemannian metric conformal to $g_e$, $(g=\phi g_{e})$, we have
\begin{align*} \label{ineq1}
m_1l_1(g)+\cdots + m_pl_p(g)&\leq\int_{S_{1}}\bar{\phi}(\gamma)d\mu(\gamma)+\cdots+\int_{S_{p}}\bar{\phi}(\gamma)d\mu(\gamma)\\
&= \int_{\Gamma}\bar{\phi}(\gamma)d\mu(\gamma)\\
&=\int_{M}\phi(u,v)d\mu^{*}(u,v)\\
&=\int_{M}\phi(u,v)dg_{e}\\
&\leq\left(\int_{M}\phi^{2}(u,v)dg_{e}\right)^{\frac{1}{2}}\left(\int_{M}dg_{e}\right)^{\frac{1}{2}}\\
&=\sqrt{vol(g)vol(g_{e})}.
\end{align*}

Using the inequality of arithmetic and geometric means we get

$$p\big(m_1\cdots m_p l_1(g)\cdots l_p(g) \big)^{\frac{1}{p}}\leq  m_1l_1(g)+\cdots+m_pl_p(g) $$

with equality if and only if 

$$m_1l_1(g)=m_2l_2(g)=\cdots=m_pl_p(g)$$

Finally, combining the two inequalities we get that, under the required conditions, the following inequality holds

\begin{equation} \label{extineq}
p^{p}\big(m_1\cdots m_p\big) l_1(g)\cdots l_p(g) \leq  \big(vol(g)vol(g_{b})\big)^\frac{p}{2},
\end{equation}

with equality if and only if $g$ is proportional to $g_{e}$.
\end{proof}

\section{Geometric Inequality of type $\sigma-v$ on the Klein bottle}

In the following, we prove geometric inequality of type $\sigma-v$. We use the notations introduced in section \ref{intro}.

\begin{proof}[Proof of theorem \ref{1}]
The inequality in the first case was proved by C.Bavard in \cite{bavard5} (corollary 2). For each $\beta < {2}\ln\left(\tan\left(\frac{\pi}{4}+\frac{b_{0}}{2}\right)\right)$ the extremal metric is 
$$G_b'=f^{2}(v)du^{2}+dv^{2}:\ |v|\leq 2b,\ |u|\leq \pi/2$$ 

where $f$ is a function invariant by the translation $v \mapsto v+2b$ and equal to $\cos $ in $|v|\leq b$. Here we have $\beta=2\ln(\tan(\frac{\pi}{4}+\frac{b}{2}))$.\\

To prove the case when ${\beta}>{2}\ln\left(\tan\left(\frac{\pi}{4}+\frac{b_{0}}{2}\right)\right)$, we endow $K$ with the metric $G_{b}$:  
\begin{equation*}
G_{b}(u,v) = \left\{ \begin{array}{l l} \cos^{2}(v)du^{2}+dv^{2} & \quad
\text{if \quad $|v|\leq\omega$ or $2b-\omega\leq|v|\leq2b$}\\ \cos^{2}(\omega)du^{2}+dv^{2} & \quad \text{if \quad $\omega\leq|v|\leq 2b-\omega $}\\ \end{array} \right.
\end{equation*}

where $b\geq b_{0}$. \\

We aim to find $\omega$ ($b_{0}\leq\omega\leq b$) so that $G_b$ verifies the conditions of theorem \ref{extcond}.
 
We consider the two families of curves $:$
\begin{equation*}
S_{1} = \left\{\gamma_{\theta}^{a}: |a|\leq \omega \textrm{ or } 2b-\omega\leq |a|\leq 2b, \theta\in \mathbb{R}/\pi \mathbb{Z}\right\},	
\end{equation*}

\begin{equation*}
S_{2} = \left\{\gamma_{u}: u\in \mathbb{R}/\pi \mathbb{Z}\right \}	
\end{equation*}
where $\gamma_{u}(t) = (u, t)$ and $|t|\leq 2b $.\\ %Let $g$ be a Riemannian metric conformal to $G_{b}$, $(g = \phi^{2}G_{b}$, where $\phi$ is a positive function defined on $K_{2b}$ when it is equipped with $G_{b}$$)$. 

We endow $S_{1}$ with the measure $\mu_{1}$ where,  
\begin{equation*}
\mu_{1} =
\left\{ \begin{array}{l l} h(a)da\otimes d\theta & \quad \text{if}\quad a\geq0\\ h(-a)da\otimes d\theta & \quad \text{if}\quad a<0\\ \end{array} \right.
\end{equation*}
On $S_{2}$, we put the measure $\mu_{2}$, where 
\begin{equation*}
\mu_{2} = \cos(\omega)du	
\end{equation*}
The condition $^*\mu_1+^*\mu_2=dG_b$ gives
\begin{equation*}
h(a) = \frac{\sin(a)}{\pi\cos(a)}\sqrt{\cos^{2}(a)-\cos^{2}(\omega)}.
\end{equation*}
(we used Abel integration equation, see \cite{bavard2} and \cite{pu} p. 66  for some details on the calculations).

Next, since $l_\sigma(G_b)=\pi$ and $l_v(G_b)=4b$ the condition $m_1l_\sigma(G_b)=m_2l_v(G_b)$ gives the relation

\begin{equation}\label{omega}
\tan(\omega)= b + \omega.	
\end{equation}

For any $b\geq {b_{0}}$, the last equation has a unique solution  $\omega\in[b_{0},\frac{\pi}{2}[$. Moreover, it can be checked easily that $\omega$ is related to ${\beta}$ by
\begin{equation*}
2\sin(\omega) = \left(\frac{\beta}{2}-2\ln\left(\tan\left(\frac{\pi}{4}+\frac{\omega}{2}\right)\right)\right)\cos^{2}(\omega)+4\omega\cos(\omega),	
\end{equation*}

Now,  the volume of $G_b$ is equal to
\begin{align*}
vol(G_{b}) &= 2\int^{\frac{\pi}{2}}_{-\frac{\pi}{2}}\int^{\omega}_{-\omega}\cos(v)dvdu	+2\int^{\frac{\pi}{2}}_{-\frac{\pi}{2}}\int^{2b-\omega}_{\omega}\cos(\omega)dvdu\\ &= 4\pi\sin(\omega)+4\pi(b-\omega)\cos(\omega).
\end{align*}

Finally, from inequality \ref{extineq} and for $\omega$ solution of equation \ref{omega} we get the inequality 
\begin{equation*}
l_v(g)l_\sigma(g)\leq\frac{1}{2\cos(\omega)}vol(g)	
\end{equation*}
which holds for every Riemannian metric $g$ conformal to $G_{b}$. The equality is achieved if and only if $g$ is proportional to $G_{b}$, and the result follows.
  
\end{proof}

We denote by $M$ be the Mobius band with boundary obtained by taking the quotient of $\mathbb{R}\times [-\beta,\beta]$ by the group generated by the map $\sigma: z\mapsto \bar{z}+\pi$. We also denote by $g_\beta$ the flat metric induced by such quotient.\\

\begin{rem} The same method gives, by considering the restriction of $G_b$ and $G_b'$ on the set $\{|v|\leq b,\ |u|\leq \pi/2\}$, the same type of inequality for the Mobius Band (\cite{blatter}, theorem 2).

\end{rem}

\begin{cor}(\cite{blatter}, Satz 2)

Let $b_{0}$ be such that $b_{0}\in ]0,\frac{\pi}{2}[$ and
$\tan(b_{0}) = 2b_{0}$, and let $\omega$ be such that
$b_{0}\leq\omega<\frac{\pi}{2}$ and $\sin(\omega) = \left(\beta-\ln\left(\tan\left(\frac{\pi}{4}+\frac{\omega}{2}\right)\right)\right)\cos^{2}(\omega)+2\omega\cos(\omega)$.

Then for every Riemannian metric $g$ on the Mobius band $M$ conformal to
$g_{\beta}$ the following holds:
\begin{enumerate}
    \item If $0<\beta\leq\ \ln\left(\tan\left(\frac{\pi}{4}+\frac{b_{0}}{2}\right)\right)$, then
       \begin{equation*}
l_\sigma(g)l_v(g)\leq\frac{\arcsin\left(\frac{e^{2\beta}-1}{e^{2\beta}+1}\right)}{\frac{e^{2\beta}-1}{e^{2\beta}+1}}vol(g).
\end{equation*}

    \item If $\beta>\ln\left(\tan\left(\frac{\pi}{4}+\frac{b_{0}}{2}\right)\right)$, then
\begin{equation*}
l_\sigma(g)l_v(g)\leq \frac{1}{2\cos(\omega)} vol(g).
\end{equation*}
\end{enumerate}
Moreover, the equality is achieved by a spherical metric in the first case and a flat-spherical one in the second case.
\end{cor}
%Since when $ b\rightarrow +\infty$ we have

%\begin{equation*}
%\frac{vol(G_{b})}{l_{\sigma}(G_{b})l_{v}(G_{b})}=\frac{4\pi\sin(\omega)+4\pi(b-\omega)\cos(\omega)}{4\pi b} \rightarrow 0.\\	
%\end{equation*}

\section{Geometric Inequality of Type $\sigma^n-v$ on the Klein bottle}

In the following, we prove geometric inequality of type $\sigma^n-v$. We use the notations introduced in section \ref{intro}.

\begin{proof}[Proof of theorem \ref{2}]
The inequality in the first case can be deduced from the proof of theorem \ref{1} (first case) since in the case $b<\frac{\pi}{3}$, $L_{\sigma}(G_b') = l_\sigma(G_b')$. It only remains to notice that the conformal type is  
\begin{equation*}
\beta = 2\ln(\tan(\frac{\pi}{4}+\frac{b}{2}))
\end{equation*}
We deduce that for $0<{\beta} \leq {2}\ln\left(2+\sqrt{3}\right)$, 
\begin{equation*}
l_{v}(g)L_{\sigma}(g)\leq\frac{\arcsin\left(\frac{e^{{\beta}}-1}{e^{{\beta}}+1}\right)}{\frac{e^{{\beta}}-1}{e^{{\beta}}+1}}vol(g).
\end{equation*}
Note that for $b> \frac{\pi}{3}$ the horizontal lines closed by the translation $\sigma^2$ become shorter than the curves $\gamma_\theta^a$ and therefore $L_\sigma $ becomes achieved by such lines.\\

To prove the inequality in the second case, we endow $K$ with the metric $H_{b}$: 
\begin{equation*}
H_{b} = \left\{ \begin{array}{l l} \cos^{2}(v)du^{2}+dv^{2} & \quad
\text{if \quad $|v|\leq\frac{\pi}{3}$ or $2b-\frac{\pi}{3}\leq|v|\leq2b$}\\ \frac{1}{4}du^{2}+dv^{2} & \quad \text{if \quad $\frac{\pi}{3}\leq|v|\leq 2b-\frac{\pi}{3}$}\\ \end{array} \right.
\end{equation*}

where $b\geq\frac{\pi}{3}$.
We then consider the three families of curves $:$
\begin{equation*}
S_{1} = \left\{\gamma_{\theta}^{a}: \theta\in \mathbb{R}/\pi \mathbb{Z}, |a|\leq \frac{\pi}{3}\textrm{ or } 2b-\frac{\pi}{3}\leq |a|\leq 2b\right\},	
\end{equation*}

\begin{equation*}
S_{1}' = \left\{\delta_a: \frac{\pi}{3}\leq |a|\leq 2b-\frac{\pi}{3}\right\},	
\end{equation*}

where $\delta_a(t)=(t,a)$, $|t|\leq \pi/2$, and
\begin{equation*}
S_{2} = \left\{\gamma_{u}(t): u\in \mathbb{R}/\pi \mathbb{Z}\right\}	
\end{equation*}
where $\gamma_{u}(t) = (u, t)$ with $|t|\leq 2b$.\\
We put on $S_{1}$ the measure $\mu_{1}$ where,  
\begin{equation*}
\mu_{1} =
\left\{ \begin{array}{l l} h(a)da\otimes d\theta & \quad \text{if}\quad a\geq0\\ h(-a)da\otimes d\theta & \quad \text{if}\quad a<0\\ \end{array} \right.
\end{equation*}
On $S_{2}$, we put the measure $\mu_{2}$, where 
\begin{equation*}
\mu_{2}= \frac{m^{'}}{\pi}du	
\end{equation*}
Finally, on $S_{1}'$, we put the measure $\mu_{3}$, where 
\begin{equation*}
\mu_{3} = \left(1-2\frac{m^{'}}{\pi}\right)da	
\end{equation*}
The condition $^*\mu_1+^*\mu_2+^*\mu_3=dH_b$ gives

\begin{equation*}
h(a) = \frac{\sin(a)}{\pi\cos(a)}.\frac{\cos^{2}(a)-\frac{m^{'}}{2\pi}}{\left(\cos^{2}(a)-\frac{1}{4}\right)^{\frac{1}{2}}}.
\end{equation*}
The mass $m_1$ of the measure on $S_{1}\cup S_{1}'$ is
\begin{align*}
m_1 &= 4\int^{\frac{\pi}{2}}_{-\frac{\pi}{2}}\int^{\frac{\pi}{3}}_{0}h(a)dad\theta + \int^{2b-\frac{\pi}{3}}_{\frac{\pi}{3}}\left(1-\frac{2m^{'}}{\pi}\right)dv\\
&= 2\sqrt{3}+2b-2\frac{\pi}{3}-\frac{4m^{'}b}{\pi},	
\end{align*}
The mass of the measure on $S_{2}$ is $m^{'}$ and $vol(H_{b})=2\pi\left(\sqrt{3}+b-\frac{\pi}{3}\right)$.\\

 We deduce from theorem \ref{extcond} (first part) that
\begin{equation*}
m_1L_{\sigma}(g)+m^{'}l_{v}(g)\leq \sqrt{2\pi\left(\sqrt{3}+b-\frac{\pi}{3}\right)vol(g)}	
\end{equation*}

which holds for every real number $m^{'}\in[0,\frac{\pi}{2}]$ and for every Riemannian metric $g$ conformal to $H_{b}$. Moreover, the equality sign is achieved for $g = H_{b}$.\\ 

Next, we have $L_{\sigma}(H_{b}) = \pi$, and $l_{v}(H_{b}) = 4b$, and then the condition $m_1L_\sigma(H_b)=m'l_v(H_b)$ gives

\begin{equation*}
m^{'} = \frac{3\sqrt{3}\pi+3b\pi-\pi^{2}}{12b} \textrm{ \ and\  } m_1 = \sqrt{3}+3b-\frac{\pi}{3}.
\end{equation*}
Note that since $m^{'}\in[0,\frac{\pi}{2}]$ we have $b \geq \frac{\pi}{3}$.\\

It follows that for every Riemannian metric $g$ conformal to $H_b$ we have
\begin{equation*}
L_\sigma(g)l_v(g)\leq \frac{2b}{\sqrt{3}+b-\frac{\pi}{3}}vol(g)	
\end{equation*}
with equality if and only if $g$ is proportional to $H_{b}$. \\
The conformal type of $H_{b}$ is 
\begin{equation*}
\beta = 2\ln(2+\sqrt{3})+4(b-\frac{\pi}{3}).	
\end{equation*}
Therefore when $b\geq\frac{\pi}{3}$, $\beta\geq 2\ln(2+\sqrt{3})$ and
the desired result follows. 
\end{proof}

\begin{rem} We can also get the same result on the Mobius band $M$ (theorem 3 in \cite{blatter}) using a minor adaptation of the previous proof.
\end{rem}

\begin{cor}(see \cite{blatter}, Satz 3)
For every Riemannian
metric $g$ on the Mobius band $M$ conformal to $g_{\beta}$,
\begin{enumerate}
    \item If $0<\beta\leq\ln\left(2+\sqrt{3}\right)$, then
       \begin{equation*}
sys(g)l_v(g)\leq\frac{\arcsin\left(\frac{e^{2\beta}-1}{e^{2\beta}+1}\right)}{\frac{e^{2\beta}-1}{e^{2\beta}+1}}vol(g).
\end{equation*}
    \item If $\beta>\ln\left(2+\sqrt{3}\right)$, then
\begin{equation*}
sys(g)l_v(g)\leq
\frac{2}{3}.\frac{3\beta+2\pi-3\ln(2+\sqrt{3})}{2\sqrt{3}+\pi\beta-\ln(2+\sqrt{3})}
vol(g).
\end{equation*}
\end{enumerate}
Moreover, the equality is achieved by a spherical metric in the first case and a flat-spherical one in the second case.
\end{cor}

\section{Geometric Inequality of type $\sigma-v-h$ on the Klein bottle}

In the following, we prove geometric inequality of type $\sigma-v-h$. We use the notations introduced in section \ref{intro}.

\begin{proof}[Proof of theorem \ref{3}]

We define on $K$ the metric
\begin{equation*}
E_{b}(u,v) = \left\{ \begin{array}{l l} \cos^{2}(v)du^{2}+dv^{2} & \quad
\text{if \quad $|v|\leq\omega$ or $2b-\omega\leq|v|\leq2b$}\\ \cos^{2}(\omega)du^{2}+dv^{2} & \quad \text{if \quad $\omega\leq|v|\leq 2b-\omega$}\\ \end{array} \right..
\end{equation*}
where, $\omega$ is a positive real number such that $\omega\leq b$.

We  consider the families of curves 
\begin{equation*}
S_{1} = \left\{\gamma_{\theta}^{a}: |a|\leq \omega \textrm{ or } 2b-\omega\leq |a|\leq 2b, \theta\in \mathbb{R}/\pi \mathbb{Z}\right\},	
\end{equation*}

\begin{equation*}
S_{2} = \left\{\gamma_{u}: u\in \mathbb{R}/\pi \mathbb{Z}\right\},	
\end{equation*}
where $\gamma_{u}(t) = (u, t)$ with $|t|\leq 2b$, and
\begin{equation*}
S_{3} = \left\{\delta_{a}\in \mathbb{R}:\omega\leq |a|\leq 2b-\omega\right\},
\end{equation*}
where, $\delta_a(t)=(t,a)$ with $|t|\leq \frac{\pi}{2}$.\\

We define on $S_{1}$ the measure $\mu_{1}$, where
\begin{equation*}
\mu_{1} = \left\{ \begin{array}{l l} h(a)da\otimes d\theta & \quad \text{if}\quad a\geq0\\ h(-a)da\otimes d\theta & \quad \text{if}\quad a<0\\ \end{array} \right..
\end{equation*}
On $S_{2}$, we define the measure $\mu_{2}$, where 
\begin{equation*}
\mu_{2} = \frac{m^{'}}{\pi}du.	
\end{equation*}
Finally, on $S_{3}$, we define the measure $\mu_{3}$, where 
\begin{equation*}
\mu_{3} = \left(1-\frac{m^{'}}{\pi\cos(\omega)}\right)da.	
\end{equation*}

here $m\in[0,\pi\cos(\omega)]$.

The condition $^*\mu_1+^*\mu_2+^*\mu_3=dE_b$ gives
\begin{equation*}
h(a) = \frac{\sin(a)}{\pi\cos(a)}.\frac{\cos^{2}(a)-\frac{m^{'}}{\pi}\cos(\omega)}{\left(\cos^{2}(a)-\cos^{2}(\omega)\right)^{\frac{1}{2}}},
\end{equation*}

Now, we calculate the masses of the measures $\mu_i$:
\begin{equation*}
m(\mu_{1})=4\int^{\frac{\pi}{2}}_{-\frac{\pi}{2}}\int^{\omega}_{0}h(a)dad\theta = 4\sin(\omega)-\frac{4m^{'}}{\pi}\omega	
\end{equation*}

\begin{equation*}
m(\mu_{2})=\int^{\frac{\pi}{2}}_{-\frac{\pi}{2}}\frac{m^{'}}{\pi}du=m^{'}	
\end{equation*}

\begin{equation*}
m(\mu_{3})=\int^{2b-\omega}_{\omega}\left(1-\frac{m^{'}}{\pi\cos(\omega)}\right)da=2\left(b-\omega\right)\left(1-\frac{m^{'}}{\pi\cos(\omega)}\right),	
\end{equation*}
and
\begin{equation*}
vol(E_{b})=4\pi\sin(\omega)+4\pi(b-\omega)\cos(\omega).	
\end{equation*}

Since $l_{\sigma}(E_{b})=\pi$, $l_{v}(E_{b})=4b$, and $l_{h}(E_{b})= 2\pi\cos(\omega)$, the condition $m(\mu_{1})l_{\sigma}(E_{b}) = m(\mu_{2})l_{v}(E_{b}) = m(\mu_{3})l_{h}(E_{b})$ gives \\

\begin{equation*}
m^{'}=\frac{\pi\sin(\omega)}{b+\omega},	
\end{equation*}

and

\begin{equation*}
m^{'}=\frac{\pi\cos(\omega)(b-\omega)}{2b-\omega}.	
\end{equation*}
Equating the two values of $m^{'}$, we get the relation
\begin{equation*}
\tan(\omega)=\frac{b^{2}-\omega^{2}}{2b-\omega}.	
\end{equation*}

Now, since $m^{'}\in[0,\pi\cos(\omega)]$, we must have
\begin{equation*}
0\leq	\frac{\pi\cos(\omega)(b-\omega)}{2b-\omega}\leq \pi\cos(\omega)
\end{equation*}
which shows that $b>\frac{\omega}{2}$ and therfore $b$ is given by 

\begin{equation*}
b=q(\omega)=\tan(\omega)+\sqrt{\tan^{2}(\omega)-\omega\tan(\omega)+\omega^{2}}.	
\end{equation*}
where, $q:]0,\frac{\pi}{2}[\rightarrow \mathbb{R}^{+}$ is an onto continuous increasing function.\\

Finally we get from $m^{'}=\frac{\pi\sin(\omega)}{b+\omega}$, together with the relation $\tan(\omega)=\frac{b^{2}-\omega^{2}}{2b-\omega}$ and theorem \ref{extcond} that for all $ b> 0$, we have

\begin{equation*}
vol(g)^{\frac{3}{2}}\geq \frac{3\sqrt{3}}{\sqrt{\pi}}.\frac{(b-\omega)\sqrt{(b-\omega)b}}{\left(b^{4}-4b\omega+\omega^{2}+\omega^{4}-2b^{2}(-2+\omega^{2})\right)^{\frac{1}{4}}(2b-\omega)}l_{\sigma}(g)l_{v}(g)l_{h}(g),	
\end{equation*}
with equality if and only if $g$ is proportional to $E_{b}$. 
\end{proof}

\emph{Acknowledgements:} The authors are grateful to St\'ephane Sabourau for useful discussions and remarks.

\end{document}